\documentclass[11pt]{amsart}

\numberwithin{equation}{section}

\usepackage{amsfonts,amsmath,amsopn,amssymb,amsthm,color,enumitem}
\usepackage{amscd}
\usepackage{graphics}
\usepackage{color}
\usepackage{esint}
\usepackage[bookmarksopen=true]{hyperref}
\usepackage{soul}

\makeatletter
\@namedef{subjclassname@2020}{\textup{2020} Mathematics Subject Classification}
\makeatother

\newtheorem{theorem}{Theorem}[section]
\newtheorem{lemma}[theorem]{Lemma}

\newtheorem{proposition}[theorem]{Proposition}

\newtheorem{corollary}[theorem]{Corollary}

\theoremstyle{definition} 
\newtheorem{remark}[theorem]{Remark}

\def\R{\mathbb{R}}
\def\RN{\mathbb{R}^N}
\def\S{\mathbb{S}}
\def\SN{\mathbb{S}^N}

\def\HN{\mathbb{H}^N}

\def\nd{\kappa}

\begin{document}

\title[
Overdetermined problems for the Poisson equation]{Overdetermined problems for the rotationally invariant Poisson equation in model manifolds}

\author{Antonio Greco}
\address{(A.~Greco)
	Dipartimento di Matematica e Informatica,
	Universit\`a degli Studi di Cagliari,
	via Ospedale 72,
	09124 Cagliari,
	Italy}
\email{greco@unica.it}

\author{Marcello Lucia}
\address{(M.~Lucia)
	Mathematics Department,
	College of Staten Island,
	City University of New York,
	Staten Island NY 10314,
	USA}
\email{marcello.lucia@csi.cuny.edu}

\author{Pieralberto Sicbaldi}
\address{(P.~Sicbaldi)
	IMAG, Departamento de An\'alisis matem\'atico,
	Universidad de Granada,
	Campus Fuentenueva,
	18071 Granada,
	Spain \& Aix Marseille Universit\'e - CNRS, Centrale Marseille - I2M, Marseille, France}
\email{pieralberto@ugr.es}

\thanks{{\bf Acknowledgements.} 
        A.\ G.\ is a member of the Grup\-po Na\-zio\-na\-le per 
l'A\-na\-li\-si Ma\-te\-ma\-ti\-ca, la Pro\-ba\-bi\-li\-t\`{a} e le lo\-ro Ap\-pli\-ca\-zio\-ni (GNAMPA) and has been supported by the research project {\em PDEs and their role in understanding natural phenomena}, CUP F23C25000080007, funded by Fon\-da\-zio\-ne di Sar\-de\-gna, annuity 2023.
M. L. has been supported by the Simons Foundation MP-TSM-00002452.
	P.\ S.\ has been supported by the FEDER-MINECO Grants PID2020-117868GB-I00 and PID2023-150727NB-I00, and the \emph{IMAG-Maria de Maeztu} Excellence Grant CEX2020-001105-M/AEI/10.13039/501100011033. 
	This work started in Spring 2024 when M. L. and P.\ S.\ were visiting the University of Cagliari. The authors want also to thank the Visiting Professor/Scientist Programme sponsored by the Re\-gio\-ne Au\-to\-no\-ma della Sar\-de\-gna.}

\keywords{Overdetermined elliptic problems. Poisson equation. Model manifolds. Space forms. Rigidity results.}

\subjclass[2020]{35R01; 35N25, 53C24.}

\pdfbookmark[2]{Abstract}{Abstract}
\begin{abstract}
We present rigidity results for overdetermined problems associated to the rotationally invariant Poisson equation
  $-\Delta_{g_\mathcal{M}} u = f(r)$ in a model manifold $\mathcal{M} = [0,S) \times_h \mathbb S^{N-1}$ with warping function $h$. 
 The variable $r$ ranges in the interval $[0,S)$, whose endpoint $S$ is positive and possibly infinite. The first part of the paper deals with
  the problem
	\[
	\begin{cases}
		-\Delta_{g_\mathcal{M}}  {u}=f(r) &\mbox{in $\Omega$},\\
		u=\varphi(r) &\mbox{on $\partial \Omega$},\\
		\frac{\partial u}{\partial \nu} = \nd(r) &\mbox{on $\partial \Omega$},
	\end{cases}
\]
where $\Omega \subset \mathcal{M}$ is a bounded domain containing the point $O \in \mathcal{M}$ corresponding to $r = 0$, $\nu$~is the exterior unit normal
vector on~$\partial \Omega$, and $f$, $\varphi$, $\nd$ are three prescribed functions. 

In the second part of the paper, we consider a similar overdetermined problem for the exterior  Bernoulli problem in a domain $\Omega \setminus \overline B_{R_0}(O)$, 
where $B_{R_0}(O)$ denotes the geodesic ball centered at~$O$ with
radius~$R_0$, within the class of functions that vanish on $\partial B_{R_0}(O)$.
In both cases, we give conditions on $f$, $\varphi$ and~$\nd$ implying that the solution $u$ is radial and $\Omega$ is a geodesic ball centered at~$O$. Our results apply in particular to the three space forms $\RN$, $\HN$ and $\SN$.
\end{abstract}

\maketitle

\section{Introduction}

Overdetermined problems, where both Dirichlet and Neumann boundary conditions are prescribed, arise naturally in mathematical physics and 
Serrin's celebrated result  \cite{serrin}  have been a turning point in this topic. 
His moving plane method, together with a refinement of Hopf's Lemma, led to the conclusion that the ball is the only bounded domain $\Omega \subset \mathbb R^N$ that permits the existence of a solution to the problem
\begin{equation} \label{eq:Serrin}
  - \Delta u = 1, 
  \quad
  u = 0 \, \hbox{ on } \partial \Omega, 
  \quad
  \frac{\partial u}{\partial \nu} = \hbox{const.} \hbox{ on } \partial \Omega.
\end{equation}
An alternative proof in $\RN$ has been given by Weinberger \cite{W} by applying the ``$P$-function method". 
If one considers a general Lipschitz continuous nonlinearity $f(u)$ instead of a constant,
Serrin's strategy still works under the additional requirement that the problem admits a {\it positive} solution
(see also \cite{GaLe,PS07} for relaxing the regularity requirements of Serrin's initial result). We stress that assuming positivity is crucial, and in fact nontrivial examples of
 sign-changing solutions to overdetermined problems
$$
  - \Delta u = f(u), 
  \quad
  u = 0 \, \hbox{ on } \partial \Omega, 
  \quad
  \frac{\partial u}{\partial \nu} = \hbox{const.} \hbox{ on } \partial \Omega,
$$
have recently been constructed in domains $\Omega$ different from a ball for some suitable nonlinearity $f(u)$ (see \cite{R}). 

\smallskip

Since then, Serrin's rigidity result has been investigated with non-constant boundary conditions or by considering domains in general manifolds. 
In this latter case, one faces some obstructions for the validity of such a result. On the $N$-dimensional sphere, for example, 
a solvable instance of~(\ref{eq:Serrin}) is readily constructed in an equatorial strip. Less trivial examples were given by Berenstein and Karlovitz \cite[Theorem~3]{Karlovitz} by considering domains of $\S^N$, $N \ge 3$, whose boundary is a suitable Clifford torus.
However, given a bounded domain $\Omega$ of the hyperbolic space $\HN$, or a domain 
 $\Omega \subset \S^N$ contained in an hemisphere, Molzon \cite{Molzon:1988} (see also \cite[Theorem~4]{molzon}) was able to show that if problem~\eqref{eq:Serrin} is solvable then $\Omega$ must be a geodesic ball
and in such a case the solution is radial (clearly one replaces $\Delta$ in~\eqref{eq:Serrin} by the Laplace-Beltrami operator $\Delta_g$, where $g$ is the metric of the manifold). In \cite[Theorem~2]{molzon}, it was shown that the conclusion still holds if one considers in a hemisphere the Poisson equation $-\Delta_{g \,} u = f(r)$ with $f(r) = N \cos r$ where $r$ is the geodesic distance to the north pole. An inspection in the proof shows that in such a case $\Omega$ is in fact a geodesic ball centered at the north pole. Here again the main tool is provided by the reflection arguments implemented in Serrin's original proof. The case of a nonlinearity $f(u)$, in place of $f(r)$, was considered  
by Kumaresan and Prajapat  \cite{KuPr} for domains in $\mathbb H^N$ or in hemispheres. They showed that if the problem
\begin{equation}\label{pr0}
    - \Delta_{g \,} u = f(u) , \quad
  u = 0 \, \hbox{ on } \partial \Omega, 
  \quad
  \frac{\partial u}{\partial \nu} = \hbox{const.} \hbox{ on } \partial \Omega,
\end{equation}
admits a {\it positive} solution, then the moving plane technique still allows to conclude that $\Omega$ must be a geodesic
ball and the solution $u$ depends only on the distance from the center of~$\Omega$. 
Furthermore, when  $f(u) =  N K u + 1$ where $K$ stands for the sectional curvature of the $N$-dimensional space form,  Ciraolo and Vezzoni (\cite{CV}) have extended
Weinberger's method and recovered
such a rigidity result by making use of appropriate $P$-functions. 
We stress that in $\SN$, positive solutions~$u$ to \eqref{pr0} with $f(u) = u^p-u$ (for some $p>1$) have been constructed in \cite{RSW2} on some domains $\Omega$ (not contained in any half-sphere) that are diffeomorphic, but not equal, to a geodesic ball. We also point out that the moving plane technique, based on the reflection with respect to all directions, can be carried on only if the ambient manifold is a space form, see again~\cite{RSW2}. 

\smallskip

Space forms can be described as a special cases of 
``model manifolds", also referred to as  ``spherically symmetric manifolds". Those are by definition smooth differentiable $N$-dimensional manifolds $\mathcal{M}$
diffeomorphic to $\RN$, $N \geq 2$, given by the quotient space $[0,S) \times \mathbb S^{N - 1}\, /  \sim$ for some $S\in (0,+\infty]$, where $\sim$ identifies all the elements of $\{\, 0 \,\} \times \mathbb S^{N - 1}$, 
endowed with the rotationally symmetric metric
\begin{equation} \label{eq:RiemannianMetric}
g_\mathcal{M} = dr^2 + h(r)^2\, g_{\mathbb S^{N-1}}\,.
\end{equation}
Here $g_{\mathbb S^{N-1}}$ denotes the canonical metric of the $(N-1)$-dimensional unit spherical surface~$\mathbb S^{N-1}$ and $h(r)$ is a smooth function over the interval $[0,S)$ satisfying 
$$
   h(r) > 0 \hbox{ for any } r > 0, 
   \quad
   h'(0) =1,
   \quad  \hbox{ and } 
   h^{(2k)}(0) = 0  \hbox{ for } k = 0,1,2,\ldots .
   $$ 
Using standard notations, we can write $\mathcal{M} = [0,S) \times_h \mathbb S^{N-1}$.
The point corresponding to the elements of $\{\, 0 \,\} \times \mathbb S^{N - 1}$ is called {\it pole} and denoted by~$O$.  The mentioned conditions on the warping function~$h$ ensure that the differentiable structure extends smoothly up to~$O$~\cite[p.~20]{Petersen}. For the basic definitions and properties of model manifolds we refer to \cite{GW, PRS}. 
Typical examples of model manifolds are the three constant-curvature space forms of dimension~$N$: namely, the Euclidean space~$\RN$, the hyperbolic space~$\HN$, and the sphere~$\SN$, which correspond to
\[
h(r) = \begin{cases}
		r &\mbox{if $\mathcal{M} = \RN$},\\
		\sinh r &\mbox{if $\mathcal{M} = \HN$},\\
		\sin r &\mbox{if $\mathcal{M} = \SN$}.
	\end{cases}
\]
Here we have $S = +\infty$ in the first two cases, and $S = \pi$ in the third.
To be precise, in the third case we should write $\mathcal{M} = \SN \setminus \{\, q \,\}$, where $q$ is the point on $\SN$ antipodal to~$O$. Slightly abusing the notation, we write $\mathcal{M} = \SN$ for simplicity.
For more details on the space forms and their metric we refer the reader to \cite{C}.

\medskip

Henceforth we consider a model manifold~$\mathcal{M}$ and a bounded $C^1$-domain $\Omega$ in $\mathcal{M}$ containing~$O$. In that setting we 
discuss overdetermined elliptic problems for the rotationally invariant Poisson equation
\begin{equation}\label{Poisson}
-\Delta u=f(r),
\end{equation}
where $\Delta = \Delta_{g_\mathcal{M}}$ is the Laplace-Beltrami operator (we drop the index $_{g_\mathcal{M}}$ for shortness).
We first discuss Serrin-type problems for the Poisson equation, and then consider the overdetermined exterior Bernoulli problem.
In both cases we will present several generalizations of results obtained so far in the literature that imply rigidity of the domain and symmetry of the solutions.

\medskip

\subsection{Serrin-type rigidity results for the Poisson equation}

In the first part of the paper, we consider the overdetermined elliptic problem
\begin{equation}\label{pr}
	\begin{cases}
		-\Delta {u}=f(r) &\mbox{in $\Omega$},\\
		u=\varphi(r) &\mbox{on $\partial \Omega$},\\
		u_\nu = \nd(r) &\mbox{on $\partial \Omega$},
	\end{cases}
\end{equation}
where $\nu$ is the unit exterior normal vector on~$\partial \Omega$ (with respect to the metric $g_\mathcal{M}$),  
the shorthand $u_\nu$ represents the normal derivative $\partial u / \partial \nu$, and $f$, $\varphi$ and $\nd$ are given functions defined on the interval $(0,S)$.

\medskip

The overdetermined problem \eqref{pr} with $f$ constant, $\varphi \equiv 0$ and $\nd$ constant in the space forms
$\RN$, $\HN$ or $\SN$, are special cases covered by the works done by Serrin, Molzon and Kumaresan-Prajapat \cite{serrin, Molzon:1988, molzon, KuPr}. 
On model manifolds whose Ricci curvature satisfy $\mathop{\rm Ric}_{\mathcal{M}} \geq (N-1)\, k\, g_{\mathcal{M}}$ for some constant $k$, 
we note that the corresponding nonlinear problem \eqref{pr0} has been investigated 
with the specific nonlinearity $f(u)=1+Nku$, 
see \cite{AFM, FR, Roncoroni} (some compatibility conditions on $h$ are also required).

\medskip

Here, setting $B_R(O) = \{\, x \in \mathcal{M} : d(x,O) < R \,\}$, where $R \in (0,S)$ and $d(x,O)$ denotes the geodesic distance,  
we investigate the overdetermined problem \eqref{pr} under the hypotheses that:
\[
O \in \Omega \subset B_R(O) \hbox{ for some } R \in (0,S),
\]
\begin{equation} \label{eq:Ipotesi2_f}
f \in C^0((0,S)), \quad \int_0^r h(t)^{N-1} \, |f(t)| \, dt < +\infty\,
\quad r \in (0,S),
\end{equation}
\[
\varphi \in C^1((0,S)), \quad \nd \in C^0((0,S)).
\]
By a solution to  \eqref{pr},  we mean a function $u \in L^1(\Omega) \cap C^2(\Omega \setminus \{\, O \,\}) \cap C^1(\overline \Omega \setminus \{\, O \,\})$ which satisfies the equation in the distributional sense. Consequently, $u$ is also a classical solution in $\Omega \setminus \{\, O \,\}$. However, the left-hand side~$f$ and the solution~$u$ are allowed to have a mild singularity at~$O$.
To state our result, we introduce the following function 
\begin{equation}\label{v}
v(r) :=
{} - \frac{1}{h(r)^{N-1}} \int_0^r h(t)^{N-1} \, f(t) \, dt,\,
\quad r \in (0,S).
\end{equation}
As discussed in Proposition~\ref{derivative}, this function represents the normal derivative on~$\partial B_r(O)$ of any radial solution to the rotationally invariant Poisson equation~(\ref{Poisson}) in~$\mathcal{M}$, and furthermore we have the following property:

\medskip

\begin{quote}
The overdetermined problem\/~\eqref{pr} admits a radial solution in the geodesic ball $\Omega=B_{r_0}(O)$ if and only if $v-\nd$ vanishes at $r_0$.
\end{quote}

\medskip

\noindent Our first rigidity  theorem can be seen as a Serrin-type result and reads as follows:
\begin{theorem}\label{th1}\pdfbookmark[2]{Theorem~\ref{th1}}{Theorem1}
Assume that \/~\eqref{pr} admits a solution~$u$, and suppose that
\begin{equation}\label{varphi}
\varphi' - v \geq 0 \qquad \mbox{in $(0,S)$.}
\end{equation}
Then, $v - \nd$ has a zero in $(0,S)$. Furthermore, the solution~$u$ is radial if the following 
condition is satisfied:
\begin{enumerate}
\def\labelenumi{\rm(\arabic{enumi})}
\item 
$v-\nd \leq 0$ in\/ $(0,r_0]$ for every zero\/ $r_0$ of the function $v - \nd$.
\end{enumerate}
The same conclusion holds if instead of\/~$(1)$ we assume
\begin{enumerate}
\def\labelenumi{\rm(\arabic{enumi})}\setcounter{enumi}{1}
\item
$v-\nd \geq 0$ in\/ $[r_0, S)$ for every zero\/ $r_0$ of the function $v - \nd$,
and the domain~$\Omega$ satisfies the interior sphere condition at each point of the boundary.
\end{enumerate}

In either case, the domain\/ $\Omega$ must be a geodesic ball centered at\/~$O$ if at least one of the following additional conditions holds:
\begin{enumerate}\itemsep=0pt plus 3pt
\item[\rm(a)] The inequality in \eqref{varphi} is strict;
\item[\rm(b)] The difference $v - \nd$ is monotone, and $\nd \ne 0$ in $(0,S)$;
\item[\rm(c)] The difference $v - \nd$ is monotone, and the function $\int_0^r h(t)^{N-1} \, f(t) \, dt $ does not vanish identically in any interval $(a,b) \subset (0,S)$.
\end{enumerate}
\end{theorem}
\noindent When we say \textit{monotone} in conditions (b) and~(c) above, as well as in Theorem~\ref{th2}, we mean either non-increasing or non-decreasing.

\medskip

Notice that Theorem~\ref{th1} is not linked to positivity of the solution $u$ (while in the results obtained by the moving plane method such condition is required).

\medskip 

An immediate consequence of Theorem~\ref{th1} is that if~\eqref{varphi} holds, and  $v - \nd$ has no zeros in $(0,S)$, then the overdetermined problem~\eqref{pr} is unsolvable.
We also remark that either assumption~$(1)$ or~$(2)$ is needed, otherwise the result fails and counterexamples can be constructed as in \cite[Section~5]{greco_AMSA}.
For instance, let 
\[
\Omega = \left\{\, (x,y) \in \mathbb R^2 : \frac{\, x^2 \,}{\, a^2 \,} + \frac{\, y^2 \,}{\, b^2 \,} < 1 \,\right\}
\] 
for some $a > b > 0$, and let $\varphi \equiv 0$. 
Choosing then any H\"older continuous, positive function $f \in C^{0,\alpha}([0,a])$, by \eqref{v} we have  
$v(r) \le 0$ and therefore \eqref{varphi} holds true. The Poisson equation~(\ref{Poisson}) with the Dirichlet boundary condition $u = 0$ on~$\partial \Omega$ has a unique solution $u \in C^{2,\alpha}(\overline \Omega)$ (see~\cite[Theorem~6.14]{GT}), which is positive by the strong maximum principle. Since the problem is invariant under the change of variable $x \mapsto -x$, as well as $y \mapsto -y$, and by uniqueness, we must have $u(x,y) = u(\pm x, \pm y)$ for any choice of the sign in front of the variables. But then, for every $r \in [b,a]$ we may define $\nd (r) = -u_\nu(x,y) > 0$ by choosing any point $(x,y) \in \partial \Omega$ satisfying $x^2 + y^2 = r^2$, and the definition is well posed because it is independent of the choice of the specific point (there are from $2$ to $4$ admissible choices). Thus, we have constructed a solvable instance of the overdetermined problem~\eqref{pr} in the non-radial domain~$\Omega$. In this case, the function~$\nd(r)$ must fail to satisfy the requirements of Theorem~\ref{th1} because the conclusion does not hold.

\medskip\vspace{0pt plus 10pt}

Theorem~\ref{th1} extends several results obtained for very specific functions~$f$, $\varphi$, $\nd$:
 \begin{enumerate} \itemsep=4pt
\item[$\bullet$]
When $\mathcal{M} = \RN$, $f \equiv 1$ and $\varphi \equiv 0$ we recover a result obtained in~\cite{greco_AMSA} (see our Corollary~\ref{basic}). Again in the case when $\mathcal{M} = \RN$ we also note that the assumptions of Theorem~\ref{th1} hold for functions satisfying 
$$
 f (r) \geq (N+ \alpha) \, r^{\alpha} , 
 \quad
 \varphi \equiv 0, 
 \quad
 \nd (r) \geq - r^{\beta},
 \quad
 \beta - \alpha - 1 \geq 0,
 \quad
 \alpha > -N.
$$
Hence our theorem extends two results obtained by Onodera~\cite[Propositions 2.1 and~2.2]{onodera}  who treated the cases
$f (r) = (N+ \alpha) \, r^{\alpha} $ and $ \nd (r) = - r^{\beta}$ using a foliation argument.
\item[$\bullet$]
In the case when $\mathcal{M} = \SN$, instead, Theorem~\ref{th1} must be compared with~\cite[Theorem~2]{molzon} (see also \cite[Theorem~1.1]{CiVe_2017}) that deals with 
a domain~$\Omega$ of class $C^2$  included in the upper hemisphere~$\SN_+$ but not required \textit{a priori} to contain the north pole~$O$. 
For the specific functions $f(r) = N\cos r$, $\varphi \equiv 0$, and $\nd$ constant, the author proves that if the overdetermined problem~\eqref{pr} has a solution $u \in C^2(\overline \Omega)$, then $\Omega$ must be a geodesic ball centered at~$O$. The same conclusion is reached for any $\Omega \subset \SN$ provided that $u$ is positive (an extension to conical domains is obtained in~\cite{Lee-Seo_MN}). Note that under such conditions we have $v(r) = -\sin r$ and therefore we may use Theorem~\ref{th1} to deal with the case where $u$ is not supposed to be positive, provided $\SN_+ \subset \Omega$: see our Corollary~\ref{further}.
\item[$\bullet$]
In general, as soon as there exists an isometry $F \colon (\mathcal{M},g_\mathcal{M}) \to (\mathcal{M},g_\mathcal{M})$ such that $F(O) \ne O$, the conclusion of Theorem~\ref{th1} (that $\Omega$~is a geodesic ball centered at~$O$) cannot hold for constant $f,\varphi,\kappa$ because problem~(\ref{pr}) becomes invariant under isometries: this is the case when $\mathcal{M}$ is a space form, which was considered by Molzon and by Kumaresan and Prajapat in the mentioned papers.
\end{enumerate}

\medskip
Related results have been obtained for other elliptic operators. See, for instance: \cite{greco_AMSA}~for the constant mean curvature, as well as the constant Gauss curvature equations; \cite{GrPi}~for the $p$-Laplacian $\Delta_p$ with $p \in (1,+\infty)$; \cite{greco_Symmetry}~for the infinity-Laplacian; \cite{CGM}~for the normalized $p$-Laplacian; \cite{GrMe}~for the Finsler $p$-Laplacian; \cite{Greco-Mascia}~for the fractional Laplacian. In~\cite{CiVe_2016}, the authors deal with the equation $-\Delta_{p\, } u = 1$ in any Riemannian manifold such that the Laplacian $\Delta d(x)$ of the distance function $d(x)$ from $x$ to~$O$ is also a function of $d(x)$ only. The paper that started the series is~\cite{HePhPr}, which however refers to a capacity problem (see Section~\ref{proof2}). For an existence theory see \cite{BHS} and~\cite{onodera}. The equation $-\Delta u = \delta$ (Dirac's delta function) is considered in~\cite[Theorem~1]{BaSh}.

\medskip

\subsection{Rigidity results for the Bernoulli problem}

In the second part of our paper we consider the exterior Bernoulli problem, which is a capacity problem related to \eqref{pr}. More precisely, we fix a ball $B_{R_0}(O)$ and, assuming $\overline B_{R_0}(O) \subset \Omega$, we consider the Poisson equation only in the set difference $\Omega \setminus \overline B_{R_0}(O)$. Solutions are required to vanish on~$\partial B_{R_0}(O)$. We end up with the overdetermined problem
\begin{equation}\label{pr2}
	\begin{cases}
		-\Delta u=f(r) &\mbox{in $\Omega \setminus \overline B_{R_0}(O)$},\\
		u=0 &\mbox{on $\partial B_{R_0}(O)$},\\
		u=\varphi(r) &\mbox{on $\partial \Omega$},\\
		u_\nu = \nd (r) &\mbox{on $\partial \Omega$},
	\end{cases}
\end{equation}
where, as before, $f$, $\varphi$ and $\nd$ are three prescribed functions. To be specific, we assume:
\begin{equation} \label{eq:Ipotesi3}
\overline B_{R_0}(O) \subset \Omega \subset B_R(O) \hbox{ for some } R_0,R \in (0,S),
\end{equation}
\[
f \in C^0([R_0,S)), \quad \varphi \in C^1((R_0,S)), \quad \nd \in C^0((R_0,S)) .
\]
In this case, we deal with classical solutions $u \in C^2(\Omega \setminus \overline B_{R_0}) \cap C^1(\overline \Omega \setminus \overline B_{R_0}) \cap C^0(\overline \Omega \setminus B_{R_0})$.
To state our result, for $r \in (R_0,S)$ and $c \in \mathbb R$ we introduce the function 
\begin{equation} \label{w}
w(r,c) =
\frac1{\, h(r)^{N-1} \, \int_{R_0}^r \frac{ds}{h(s)^{N-1}} \,}
\left( c - 
 \int_{R_0}^r  \int_{s}^r \Big( \frac{h(t)}{h(s)} \Big)^{N-1} \, f(t) \, dt \, ds\right) .
\end{equation}

\noindent As explained in Proposition~\ref{derivative_0}, this function represents the normal derivative on~$\partial B_r(O)$ of any radial solution to the equation~(\ref{Poisson}) in $\mathcal{M} \setminus \overline B_{R_0}(O)$, 
vanishing on~$\partial B_{R_0} (O)$ and attaining the value~$c$ on~$\partial B_r(O)$. Furthermore we have the following property:

\medskip

\begin{quote}
The overdetermined problem\/~\eqref{pr2} admits a radial solution in the geodesic annulus $B_{r_0}(O) \setminus \overline B_{R_0}(O)$ if and only if $\nd(r_0) = w(r_0,\varphi(r_0))$.
\end{quote}

\medskip

\noindent Our second rigidity result is the following:
\begin{theorem}\label{th2}\pdfbookmark[2]{Theorem~\ref{th2}}{Theorem2}
Assume that \eqref{pr2} admits a solution~$u$, and suppose that\/ $\varphi$ satisfies for all $r \in (R_0,S)$ the linear differential inequality
\begin{equation}\label{varphi2}
\varphi'(r)  - w(r,\varphi(r)) \ge 0.
\end{equation}
Then, the difference $w(\cdot, \varphi(\cdot)) - \nd(\cdot)$ has a zero in $(R_0,S)$. Furthermore, the solution~$u$ is radial if the following 
condition is satisfied:
\begin{enumerate}
\def\labelenumi{\rm(\arabic{enumi})}
\item 
$w(\cdot, \varphi(\cdot)) - \nd(\cdot) \le 0$ in $(R_0,r_0]$ for every zero~$r_0$ of the function $w(\cdot, \varphi(\cdot)) - \nd(\cdot)$.
\end{enumerate}
The same conclusion holds if instead of\/~$(1)$ we assume
\begin{enumerate}\setcounter{enumi}{1}
\def\labelenumi{\rm(\arabic{enumi})}
\item  
$w(\cdot, \varphi(\cdot)) - \nd(\cdot)  \ge 0$ in $[r_0,S)$ for every zero~$r_0$ of the function $w(\cdot, \varphi(\cdot)) - \nd(\cdot)$,
and the domain\/ $\Omega$ satisfies the interior sphere condition at each point of the boundary.
\end{enumerate}
In either case, the domain\/ $\Omega$ must be a geodesic ball if at least one of the following additional conditions holds:
\begin{enumerate}\itemsep=0pt plus 3pt
\item[\rm(a)] The inequality in \eqref{varphi2} is strict;
\item[\rm(b)] The difference $w(\cdot, \varphi(\cdot)) - \nd(\cdot)$ is monotone, and $\nd \ne 0$ in $(R_0,S)$;
\item[\rm(c)] The difference $w(\cdot, \varphi(\cdot)) - \nd(\cdot)$ is monotone, and the function $f$ does not vanish identically in any interval $(a,b) \subset (R_0,S)$.
\end{enumerate}
\end{theorem}

A special case when the assumptions are satisfied occurs when $\mathcal{M} = \RN$, $f \equiv 0$, $\varphi$ is a positive constant and $\nd$ is a constant. Such a case is also a corollary of a theorem by Reichel \cite{reichel} (see \cite{HePhPr,Sirakov} as well), where the author generalized the moving plane method to annular domains for the overdetermined elliptic problems
\[
	\begin{cases}
		-\Delta u = f(u) &\mbox{in $\Omega \setminus \overline{B}_{R_0}(O)$},\\
		u=0&\mbox{on $\partial B_{R_0}(O)$},\\
		u=a>0&\mbox{on $\partial \Omega$},\\
		u_\nu = \textnormal{const.} &\mbox{on $\partial \Omega$},
	\end{cases}
\]
under the assumption that $f$ is the sum of a Lipschitz continuous and a non-decreasing function. Up to our knowledge, \cite{HePhPr,reichel,Sirakov} have not been generalized to $\HN$ and $\SN$ but we believe that the mentioned result should still hold in such cases, again under the hypothesis that the annular domain stays in the upper hemisphere if the ambient space is $\SN$. Nontrivial solutions to overdetermined elliptic problems in annular domains of $\RN$ with locally constant Dirichlet and Neumann conditions have been constructed in \cite{ABM, EFRS, KS}.

\medskip

Another special case of Theorem \ref{th2} is a result of~\cite{greco_MMAS}, where the case $\mathcal{M} = \RN$, $f \equiv 0$, $\varphi$ is a positive constant is considered under the assumption that the product $r^{N-1} \, \nd(r)$ is non-decreasing (see also our Corollary~\ref{annulus}).

\medskip

We remark that when $\mathcal{M} = \RN$ or $\HN$ the function $h(r)$ is increasing, which allows to apply Theorem \ref{th2} to a large class of functions $\nd$: see Lemma~\ref{explanation} for an explanation. The same remark holds in $\SN$ if the domain $\Omega \setminus \overline B_{R_0}(O)$ stays in the upper hemisphere, i.e. when $h(r)=\sin (r)$ and $0 < r < \pi/2$. If the ambient manifold is $\SN$ but we do not know a priori that the domain $\Omega \setminus \overline B_{R_0}(O)$ stays in the upper hemisphere, then the hypothesis on $\nd$ is more demanding (for example constant functions $\nd$ are not suitable when $f \equiv 0$). Again, this is according to the fact we mentioned before that in $\SN$ rigidity results via the moving plane method can be obtained only for domains contained in a hemisphere.

\medskip

The literature on the Bernoulli problem in $\RN$ is quite extended. The symmetry theorem we present here is related to~\cite{HePhPr}, whose results were refined in~\cite{greco_MMAS} (see also~\cite{reichel}). In particular, we show that such results hold in any model manifold and for every continuous~$f$. Related results for the quasilinear elliptic equation $-\mathop{\rm div}(g(|\nabla u|) \, \nabla u) = f(r,u)$ are found in~\cite[Theorem~1.3]{greco_JAM}. Approximate radial symmetry (stability) was investigated in~\cite{HePh} (see also \cite[Theorem~5.1]{greco_JAM}).
Further overdetermined problems in annular domains on space forms are considered in~\cite{Lee-Seo_JIA}. 


\section{Serrin-type rigidity result} \label{sec:rel_dic_prob}

This section is devoted to the proof of Theorem \ref{th1}, and some of its consequences.

\subsection{A Dirichlet problem in the ball}

Before giving the proof of Theorem~\ref{th1} we collect some preliminary results on the radial solutions for the Dirichlet problem
\begin{equation}\label{Dp}
	\begin{cases}
		-\Delta {u}=f(r) &\mbox{in $B_R(O)$},\\
		u=c &\mbox{on $\partial B_R(O)$},
	\end{cases}
\end{equation}
where $c \in \R$. To this end, we observe that the Laplace-Beltrami operator $\Delta$ associated to the metric~\eqref{eq:RiemannianMetric} may be written as
\[
\partial_r^2 + (N-1) \, \frac{h'(r)}{h(r)} \, \partial_r + \frac{1}{h(r)^2}\, \Delta_{\S^{N-1}},
\quad
\hbox{ for } r \in (0,S) ,
\]
where $\Delta_{\S^{N-1}}$ is the Laplace-Beltrami operator on~$\S^{N-1}$.
In particular, in the class of radial solutions, the PDE in \eqref{Dp} reads as
\begin{equation}\label{ode}
u'' + (N-1) \, \frac{h'(r)}{h(r)} \, u' = -f(r),
\quad
\hbox{ for } r \in (0,R) ,
\end{equation}
which may be equivalently rewritten as
\begin{equation}\label{rewritten}
-\Big( h(r)^{N - 1} \, u'(r) \Big)' = h(r)^{N - 1} \, f(r).
\end{equation}
With this in mind, some properties of radial solutions are summarized in the following proposition.
\begin{proposition}\label{derivative}
Consider\/ $f$ satisfying\/~{\rm(\ref{eq:Ipotesi2_f})} and\/ $c \in \mathbb R$.
The Dirichlet problem~\eqref{Dp} is solvable  (in the sense of distributions). Its solution $u = u_R(r)$ belongs to $C^2(\overline B_R(O) \setminus \{\, O \,\})$, is unique, radial, and satisfies~\eqref{ode}.
Its derivative\/ $v(r) = u_R'(r)$ is determined by~\eqref{v} which is independent of $R$, and $u_R (r)$ itself is given by
\begin{equation} \label{eq:UItself}
        u_R (r) = c + \int_r^{R} \int_0^s  \Big( \frac{h(t)}{h(s)} \Big)^{N-1}\, f(t) \, dt\, ds\,
        \quad
       \hbox{ for } r \in (0,R).
\end{equation}
Furthermore, $u_R(r)$ admits an extension satisfying~\eqref{ode} for $r \in (0,S)$.
\end{proposition}

\begin{proof}
Uniqueness follows from the comparison principle. Indeed, the difference of any two distributional solutions is harmonic in the sense of distributions, and therefore it is harmonic in the classical sense by Weyl-Caccioppoli's lemma: see \cite[pp.~4-5]{Caccioppoli} and \cite[Lemma~2]{Weyl}.
Uniqueness, together with the invariance under rotations of problem~\eqref{Dp} implies that any solution must be radial. Passing to polar coordinates, the equation~(\ref{Poisson}) for $r > 0$ becomes~\eqref{rewritten}. Since the function in~(\ref{eq:UItself}) satisfies
\[
        u_R(r) = c - \int_r^{R} v(s) \, ds\,
        \quad
       \hbox{ for } r \in (0,R),
\]
a straightforward computation shows that $u_R \in C^2((0,S))$ and satisfies~(\ref{rewritten}). In order to prove that $u = u_R(r)$ is a distributional solution of~(\ref{Poisson}), it is enough to bound its possible singularity at~$O$. To this aim, we first observe that
\[
\lim_{s \to 0^+} h(s)^{N - 1} \, v(s) = 0.
\]
This is an immediate consequence of the convergence of the integral in assumption~(\ref{eq:Ipotesi2_f}). Since $h'(0) = 1$ by assumption, we may write $v(s) = o(s^{1-N})$ as $s \to 0^+$. Consequently, using de l'H\^opital rule we check that
\[
u_R(r) =
\begin{cases}
o(\log r), &\mbox{if $N = 2$};
\\
\noalign{\medskip}
o(r^{2-N}), &\mbox{if $N \ge 3$}.
\end{cases}
\]
These bounds imply that $u$~is a distributional solution of~(\ref{Poisson}).
\end{proof}

\begin{remark}
Consider\/ $f$ satisfying\/~{\rm(\ref{eq:Ipotesi2_f})}, and let $\varphi$ and~$\nd$ be any real-valued functions over the interval $(0,S)$. By the preceding proposition we immediately see that the overdetermined problem \eqref{pr} has a radial solution in the geodesic ball $B_{r_0}(O)$ for some $r_0 \in (0,S)$ if and only if $(v-\nd)(r_0) = 0$.
\end{remark}

\subsection{Radial symmetry for Serrin-type problems}

In order to study the radial symmetry of solutions to the overdetermined problem \/~\eqref{pr}, we consider the two geodesic balls 
$B_{R_1}(O) \subset \Omega \subset B_{R_2}(O)$, whose radii are defined by
\begin{equation} \label{eq:Min-Max}
R_1 := \min_{\partial \Omega} r > 0\,, \qquad R_2 := \max_{\partial \Omega} r < S ,
\end{equation}
and introduce the unique solutions $u_1, u_2 \in C^2(\mathcal{M} \setminus \{\, O \,\})$ of
\begin{equation} \label{eq:Ostia}
   - \Delta u_i = f \hbox{ in } \mathcal{M}, 
   \quad
   u_i = u(P_i) = \varphi(R_i) \hbox{ on } \partial B_{R_i} (O), 
\end{equation}
where $P_i$ is any point on $\partial B_{R_i}(O) \cap \partial \Omega$ (see Proposition~\ref{derivative}). Since $u_1,u_2$ are radial, slightly abusing the notation we will consider them as functions of~$x$ as well as~$r$. We will make use of the following equivalence:
%
%
\begin{lemma} \label{lem:pasolini}
Let $u$ be a solution of\/~{\rm(\ref{Poisson})} in\/~$\Omega$ attaining the value~$\varphi$ at the boundary, and consider the solution $u_i$ of\/~\eqref{eq:Ostia}. Then
\begin{equation} \label{eq:MalTesta}
  u \hbox{ is radial} 
  \, \,  \, \Longleftrightarrow \, \, \, u  \equiv u_i \hbox{ in\/ $\Omega \setminus \{\, O \,\}$ for both\/ $i = 1,2$}. 
\end{equation}
\end{lemma}

\begin{proof}
Let us first point out that if a radial harmonic function $\omega$ defined in a ball $B_{R} (O)$ for some $R>0$ has a zero at some $P \in \partial B_{R} (O)$, it must vanish on $\partial B_R (O)$ 
and by the maximum principle it must vanish in all of~$B_R(O)$.  
\medskip \\
Now assume that $u$ is radial. 
By applying this previous observation with the radial harmonic function $\omega = u - u_1$ which admits a smooth extension up to the origin), we have
$u \equiv u_1$ in $B_{R_1} (O) \setminus \{\, O \,\}$ (since $(u - u_1) (P_1) = 0$). This implies $u \equiv u_1$ in $\overline \Omega \setminus \{\, O \,\}$ (unique continuation property). 
Furthermore, observing that $(u_1 - u_2)(P_2) = (u - u_2 ) (P_2) = 0$ it follows that the harmonic function $\omega = u_1 - u_2$ (which again admits a smooth extension up to the origin) vanishes in the entire ball $B_{R_2} (O)$ 
and therefore in the domain $\Omega$.
\end{proof}

The second part of Theorem \ref{th1} will be derived from a more general fact, that is stated in the proposition that comes next.

\begin{proposition}\label{ball}
If\/ $u$ is a radial solution of\/ \eqref{pr} and either assumption\/~{\rm(a)}, or assumption\/~{\rm(b)}, or\/~{\rm(c)} in Theorem~\ref{th1} is satisfied, then\/ $\Omega$~must be a geodesic ball centered at~$O$.
\end{proposition}

\begin{proof} 
Let $R_1, R_2$ be as in \eqref{eq:Min-Max} and $u_1, u_2$ be defined in \eqref{eq:Ostia}. The fact that $u$ is radial is equivalent by \eqref{eq:MalTesta} to the fact that $u \equiv u_1 \equiv u_2$ holds in~$\overline \Omega \setminus \{\, O \,\}$. Consequently we have $\varphi(r) = u_1(r)$ in $[R_1,R_2]$. Since $u_1' = v$, assumption~(a) allows the difference $\varphi - u_1$ to have at most one zero in $[R_1,R_2]$, and therefore $R_1 = R_2$ which shows that
$B_{R_1} (O) = \Omega = B_{R_2} (O)$.
\goodbreak
\medskip
Let us now treat the case when either assumption~(b) or~(c) holds. Observe first that showing $B_{R_1} (O) = \Omega = B_{R_2} (O)$ is equivalent
to proving $\nu \equiv \frac\partial{\partial r} $ along $\partial \Omega$. Indeed, if $\partial \Omega$ is the boundary of a geodesic ball centered at $O$, then Gauss Lemma~\cite[Theorem~6.9]{Lee} ensures that 
$\nu \equiv \frac{\partial}{\, \partial r \,} $.  Conversely, taking a local parametrization $\psi: U \subset \mathbb R^{N-1} \to \partial \Omega $, we derive the following information on the directional derivative of 
$r \circ \psi$ at each $y \in U$ in the direction $\xi \in \mathbb R^{N-1}$: 
$$
   D( r \circ \psi)_{(y)}  (\xi) = \frac{\partial}{\, \partial r \,} \cdot D \psi_{(y)} (\xi) = \nu \cdot D \psi_{(y)} (\xi) =  0.
$$
Hence, the distance function from $\partial \Omega$ to $O$ is locally constant. Therefore, each component of $\partial \Omega$ is a geodesic sphere, and we may write
$\partial \Omega = \textstyle\bigcup\limits_{j=1}^m \partial B_{r_j}(O)$
for convenient $r_m > \ldots > r_1 > 0$. Since $\Omega$ is connected and contains $O$, we must have $m = 1$. Consequently, $\Omega$ is a geodesic ball and so $B_{R_2} (O) = \Omega = B_{R_1} (O)$. 
\medskip \\
To prove the proposition, we recall that $u \equiv u_1 \equiv u_2$ is radial and so its gradient satisfies 
\begin{equation} \label{gradient1}
\nabla u(x) = v(r) \, \frac\partial{\, \partial r \,},
\quad
\hbox{for all } 
x \in \partial \Omega.
\end{equation} 
Of course, we also have
\begin{equation} \label{gradient2}
\nabla u(x) \cdot \nu = \nd(r),
\quad
\hbox{for all } 
x \in \partial \Omega.
\end{equation}
Furthermore, at each point $P_i$ we have $\nu = \frac{\partial}{\partial r}$, implying that the normal derivative of~$u$ coincides with the radial derivative, i.e.\
$( v - \nd) (R_i) = 0$. Consequently, $v - \nd$ vanishes identically in $[R_1,R_2]$ since this function is assumed to be monotone. Moreover, plugging into~(\ref{gradient2}) the expression of~$\nabla u$ given by~(\ref{gradient1}), we obtain
\begin{equation} \label{eq:Klug}
v(r) \Big(\frac\partial{\, \partial r \,} \cdot \nu - 1 \Big) =0
\quad
\hbox{for all } 
x \in \partial \Omega.
\end{equation}
Now, if (b) holds, then $v = \nd \ne 0$ in $[R_1,R_2]$ and from~(\ref{eq:Klug}) we get $\frac\partial{\, \partial r \,} \cdot \nu = 1$ on~$\partial \Omega$. Hence $\nu \equiv \frac\partial{\, \partial r \,}$, and the conclusion follows.

To prove the claim under assumption~(c) we argue by contradiction. Assume there exists $P \in \partial \Omega$ such that $(\nu - \frac\partial{\, \partial r \,} )(P) \not = 0$, then  we must have $R_1 < R_2$ and the inequality
 $\nu - \frac\partial{\partial r}  \not = 0$ must hold in a neighborhood $N(P) \subset \partial \Omega$, hence $\frac\partial{\, \partial r \,} \cdot \nu < 1$ there. Relation~\eqref{eq:Klug} shows that this can only occur if $v(r) = 0$ in $N(P)$.
 So by~(\ref{v}) there exists an interval $(a,b) \subset (R_1, R_2)$ such that
\[
\frac{1}{h(r)^{N-1}} \int_0^r h(t)^{N-1} \, f(t)\, dt = 0
\]
for all $r \in (a,b)$ in contradiction with assumption~(c). Hence, $\nu -  \frac\partial{\partial r} \equiv 0 $ which proves the claim.
\end{proof}

We are now in position to prove Theorem \ref{th1}.

\begin{proof}[Proof of Theorem \ref{th1}]
Consider the two functions $u_1, u_2$ defined in \eqref{eq:Ostia}.
Since $u_i' \equiv v$, the assumption \eqref{varphi} gives $\varphi' - u_i' = \varphi' - v \geq 0 $ and therefore 
$\varphi - u_i$ is increasing. Consequently, since $u_i(R_i)=\varphi(R_i)$ (by the definition of $u_i$), we must have 
$u_1 \leq u \leq u_2$
on the boundary $\partial \Omega$, and by the maximum principle this implies 
\begin{equation} \label{eq:u1-u-u2}
u_1 \leq u \leq u_2
\end{equation}
in $\Omega \setminus \{\, O \,\}$. 
Applying now the strong maximum principle together with \eqref{eq:u1-u-u2} and the equivalence noted in~\eqref{eq:MalTesta}, we are led to the following alternative: 
\begin{equation} \label{eq:Alternativa} 
\left\{
\begin{array}{l}
\hbox{either $u_1 \equiv u \equiv u_2 $ in $\Omega \setminus \{\, O \,\}$, and $u$ is radial};
\vspace{1mm} \\
\hbox{or $u_1 < u <  u_2 $ in $\Omega \setminus \{\, O \,\}$, and $u$ is not radial}.
\end{array}
\right.
\end{equation}

To continue with our analysis, we observe that the first of the inequalities~\eqref{eq:u1-u-u2} together with the fact that $(u - u_1) (P_1)= 0$ gives
\begin{equation}\label{first}
v(R_1) = u_1'(R_1) \geq u_\nu(P_1) = \nd (R_1).
\end{equation}
Similarly we also get
\begin{equation}\label{second}
v(R_2) = u_2'(R_2) \leq u_\nu(P_2) = \nd (R_2).
\end{equation}

\medskip

The relations \eqref{first} and \eqref{second} imply
\[
   ( v - \nd) (R_1) \geq 0, 
   \quad
    (v - \nd ) (R_2) \leq  0,
\]
and in particular we deduce that $(v - \nd)(r_0) = 0$ for some $r_0 \in [R_1, R_2]$. 

\medskip

Now, on account of the definition of the ball $B_{R_1}(O)$, the interior sphere condition is satisfied at the point $P_1$, 
so using the Hopf's Lemma we can rephrase the alternative \eqref{eq:Alternativa} as follows:

\begin{enumerate} \itemsep=5pt
\item[{\bf (i)}]
either $( v - \nd) (R_1) =0$, and in this case $u_1 \equiv u \equiv u_2$ in~$\overline \Omega \setminus \{\, O \,\}$;
\item[{\bf (ii)}]
or $(v - \nd) (R_1) > 0 \geq (v - \nd) (R_2)$ (and $u$ is not radial).
\end{enumerate} 

\medskip
Under the assumption (1), the second case cannot occur. Thus, $u \equiv u_1 \equiv u_2$, and so the solution $u$ is radial in $\Omega \setminus \{\, O \,\}$. 

\medskip

Under the further requirement that  the domain $\Omega$ satisfies the interior sphere condition at each point of the boundary $\partial \Omega$, we can apply Hopf's Lemma at both points $P_1, P_2$, 
and by taking into account \eqref{eq:Alternativa}, we are led to the alternative: 
 \begin{enumerate}   \itemsep=4pt
\item[{\bf (i)}]
either $( v - \nd) (R_1) = ( v - \nd) (R_2) = 0$, and in this case $u_1 \equiv u \equiv u_2$ in~$\overline \Omega \setminus \{\, O \,\}$;
\item[{\bf (ii)}]
or $(v - \nd) (R_1) > 0 >  (v -\nd) (R_2)$ (and $u$ is not radial).
\end{enumerate} 
\medskip

Under assumption (2), the second case cannot hold and we deduce as before that $u \equiv u_1 \equiv u_2$, so $u$ is radial.
\medskip

The second part of Theorem \ref{th1} is an immediate consequence of Proposition~\ref{ball}.
\end{proof}

\subsection{Applications}

In the case when $\mathcal{M} = \RN$, the next result follows by combining Theorem~2.1 of~\cite{greco_AMSA} with Example~1 in Section~3 of that reference. More generally, we have:
\begin{corollary}\label{basic}
Let\/ $\Omega \subset \mathcal{M}$ be a $C^1$ domain containing~$O$ and included in $B_R(O)$ for some $R < S$. Suppose that the ratio $\nd (r)/h(r)$ is monotone non-increasing in $r \in (0,+\infty)$. If the overdetermined problem
\begin{equation}\label{special}
	\begin{cases}
		-\Delta u=N \, h'(r) &\mbox{in $\Omega$},\\
		u=0 &\mbox{on $\partial \Omega$},\\
		u_\nu = \nd(r) &\mbox{on $\partial \Omega$}
	\end{cases}
\end{equation}
has a (classical) solution $u \in C^2(\Omega) \cap C^1(\overline \Omega)$, then\/ $u$ is radial and\/ $\Omega$ is a geodesic ball centered at~$O$.
\end{corollary}
\begin{proof}
Problem~(\ref{special}) is a particular instance of Problem~(\ref{pr}) with 
$f (r) \allowbreak = \allowbreak N \allowbreak \, h'(r)$ and $\varphi \equiv 0$. Furthermore, in this case the definition~\eqref{v} gives $v(r) = -h(r)$. 
Hence,  $ \varphi' (r) \allowbreak - \allowbreak v(r) >0 $ which establishes that both \eqref{varphi} and the stronger assumption (a) of Theorem~\ref{th1} hold.
To conclude, it remains to check the assumption (1) of Theorem~\ref{th1}.  Writing
\[
v(r) - \nd (r) = h(r) \, \Big( - 1 - \frac{\, \nd (r) \,}{h(r)} \Big),
\]
the hypothesis that the ratio $\nd (r)/h(r)$ is monotone non-increasing implies that the term in brackets is non-decreasing, hence it is non positive on $(0,r_0]$ for each zero $r_0$  of $v - \nd$.
Therefore Theorem~\ref{th1} can be applied and gives the conclusion that $u$ is radial on a domain $\Omega$ that must be a geodesic ball centered at~$O$.
\end{proof}

Note that Problem~\eqref{special} with $\nd(r) := \nd_0$ a negative constant is not covered by our Theorem~\ref{th1} as long as $h'(r) > 0$ since in that case
$v - \nd = - h(r) - \nd_0$ does not satisfy hypothesis (1) nor~(2).
By contrast, if $h'(r) < 0$ in an interval $(R_0,R)$ with $0 < R_0 < R < S$, Theorem~\ref{th1} allows to derive the following result, which is new even on the sphere $\SN$ for 
 domains~$\Omega$ containing the upper hemisphere~$\SN_+$.
To see this, it is useful to make the following observation:

\begin{remark} \label{rem:Beijing}
If in Theorem \ref{th1} we consider $C^1$ domains $\Omega$ satisfying
$B_{R_0} (O) \subset \Omega \subset B_R(O)$ for fixed geodesic balls with radii $R_0 < R < S$, then the proof goes through without modifications 
provided that $\varphi, \nd$ are just defined and satisfy the assumptions on the interval $[R_0, R]$ instead of $(0, S)$.
\end{remark}

\begin{corollary}\label{further}
Suppose that\/ $h'(r) < 0$ in an interval $(R_0,R) \subset \subset (0,S)$, and that $(R_1, \allowbreak R_2) \subset (R_0,R)$, where $R_1,R_2$ are given by~\eqref{eq:Min-Max}.
Consider a constant $\kappa_0 \in \mathbb R$ and a function $\varphi \in C^1 ( [R_0, R] )$ such that $\varphi'(r) \ge -h(r)$ for all $r \in [R_0,R]$. If the overdetermined problem
\[
	\begin{cases}
		-\Delta u=N \, h'(r) &\mbox{in $\Omega$},\\
		u=\varphi(r) &\mbox{on $\partial \Omega$},\\
		u_\nu = \kappa_0 &\mbox{on $\partial \Omega$}
	\end{cases}
\]
has a (classical) solution $u \in C^2(\Omega) \cap C^1(\overline \Omega)$, then\/ $u$ is radial and\/ $\Omega$ is a geodesic ball centered at $O$.
\end{corollary}

\begin{proof}
Since $(R_1, \allowbreak R_2) \subset (R_0,R)$, as emphasized in Remark~\ref{rem:Beijing} the conclusion of Theorem~\ref{th1} holds 
if $\varphi$ and $\nd$ satisfy each stated hypothesis only on the interval $[R_0,R]$.

We define $f(r) = N \, h'(r)$ for $r \in (0,S)$, and from~(\ref{v}) we obtain $v(r) = -h(r)$. Hence (\ref{varphi}) holds true by assumption. 
Define also $\nd(r) = \kappa_0$ for $r \in [R_0,R]$, so $v(r) - \nd (r) = - h(r) - \kappa_0$ in $[R_0,R]$. Such a function is strictly increasing in $[R_0,R]$, and $f(r)$ may vanish at the endpoints only. Hence the conclusion follows immediately from Theorem~\ref{th1}, case~(1)-(c).
\end{proof}

Note that the constant $\kappa_0$ in Corollary~\ref{further} may be replaced by any function $\nd (r)$ such that  $- h(r) - \nd(r)$ is non-decreasing in $[R_0,R]$. 

\section{Rigidity result for the Bernoulli problem}
\label{proof2}

This section is devoted to the proof of our second result, Theorem \ref{th2}, and some of its consequences.

\subsection{A Dirichlet problem in the annulus and a comparison result}

We start by highlighting some properties of the (radial) solutions of the Dirichlet problem
\begin{equation}\label{Dp2}
	\begin{cases}
		-\Delta {u}=f(r) &\mbox{in $B_R(O) \setminus \overline B_{R_0}(O)$},\\
		u=0 &\mbox{on $\partial B_{R_0}(O)$},\\
		u=c &\mbox{on $\partial B_R(O)$} ,
	\end{cases}
\end{equation}
where $c$ is a constant.
The following proposition is analogous to Proposition~\ref{derivative}.

\begin{proposition}\label{derivative_0}
Consider\/ $f \in C^0([R_0,S))$ and\/ $c \in \mathbb R$.
The Dirichlet problem~\eqref{Dp2} is solvable in $C^2(\overline B_R(O) \setminus B_{R_0}(O))$ and its solution $u = u_{R,c}(r)$ is unique, radial, and satisfies~{\rm(\ref{ode})} 
for $r \in [R_0,R]$. Its derivative is given by
\begin{align}\label{relation}
u_{R,c}'(r) &=
\frac1{\, h(r)^{N-1} \, \int_{R_0}^R \frac{ds}{h(s)^{N-1}} \,}
\left( c -
 \int_{R_0}^R  \int_{s}^r \Big( \frac{h(t)}{h(s)} \Big)^{N-1} \, f(t) \, dt \, ds\right) .
\end{align}
Furthermore, $u_{R,c}(r)$ admits an extension satisfying ~\eqref{ode} for $r \in [R_0,S)$ and the linear differential equation
\begin{equation}\label{key}
u_{R,c}'(r) = w(r,u_{R,c}(r)) ,
\quad
\hbox{ for } r \in (R_0,S) ,
\end{equation}
where $w$ is the two-variable function in~{\rm(\ref{w})}.
\end{proposition}

\begin{proof}
The claim follows by integration of the ODE~\eqref{rewritten}. A first integration over the interval $[R_0,r]$ yields
\begin{equation}\label{integrale}
h(r)^{N-1} \, u'(r)
= h(R_0)^{N-1} \, u'(R_0) 
-
\int_{R_0}^r h(t)^{N-1} \, f(t) \, dt
\end{equation}
whence
\begin{align}
\label{intermediate}
u(r) &= \int_{R_0}^r u'(s) \, ds \\
 &=
h(R_0)^{N-1} \, u'(R_0) \, \int_{R_0}^r \frac{ds}{h(s)^{N-1}}
\nonumber
- \int_{R_0}^r
\frac{ds}{h(s)^{N-1}} \int_{R_0}^s h(t)^{N-1} \, f(t) \, dt 
.
\end{align}
If we now let $r = R$ in the last equality, since $u(R) = c$ we get the following representation of $h(R_0)^{N-1} \, u'(R_0)$:
\[
h(R_0)^{N-1} \, u'(R_0) =
\frac1{\, \int_{R_0}^R \frac{ds}{h(s)^{N-1}} \,}
\left( c + \int_{R_0}^R \frac{ds}{h(s)^{N-1}}
\int_{R_0}^s h(t)^{N-1} \, f(t) \, dt \right)
\]
and plugging this into~\eqref{integrale} 
\begin{align*}   
h(r)^{N-1} \, u'(r) =
\frac1{ \int_{R_0}^R \frac{ds}{h(s)^{N-1}}}
\Big( c &+
 \int_{R_0}^R  \int_{R_0}^s \Big( \frac{h(t)}{h(s)} \Big)^{N-1} \, f(t) \, dt \, ds
\\
\noalign{\medskip}
&-  \int_{R_0}^R  \int_{R_0}^r  \Big( \frac{h(t)}{h(s)} \Big)^{N-1} \, f(t) \, dt \, ds \Big)
\end{align*}
from which we obtain~\eqref{relation} for all $r \in [R_0,S)$. Finally, for $r \in (R_0, \allowbreak S)$ we deduce from~(\ref{intermediate})
\[
h(R_0)^{N-1} \, u'(R_0) =
\frac1{\, \int_{R_0}^r \frac{ds}{h(s)^{N-1}} \,}
\left( u(r) + \int_{R_0}^r \frac{ds}{h(s)^{N-1}}
\int_{R_0}^s h(t)^{N-1} \, f(t) \, dt \right)
\]
and plugging this into~\eqref{integrale} we arrive at~\eqref{key}.
\end{proof}

\begin{remark}
Consider\/ $f \in C^0([R_0,S))$, and any functions $\varphi$ and~$\nd$ on the interval $(R_0, \allowbreak S)$. By the preceding proposition we immediately see that the overdetermined problem \eqref{pr2} 
has a radial solution in the geodesic annulus $B_{r_0}(O) \setminus \overline B_{R_0}(O)$ for some $r_0 \in (R_0,S)$ if and only if\/ $w(r_0,\varphi(r_0)) - \nd(r_0) = 0$.
\end{remark}

\goodbreak

We also need a comparison principle between a super- and a subsolution for first order ordinary differential equations, that we state as follows:

\begin{lemma}
Consider $g \colon (a,b) \times \mathbb R \to \mathbb R$, and two  functions  $\overline y, \underline y$ differentiable in $(a,b)$ satisfying:
\begin{equation}\label{Cauchy-Comparison}
\overline y' (r) \ge g(r, \overline y (r)),
\qquad
\underline y'(r) \le g(r, \underline y(r) ).
\end{equation}
for $r \in (a,b)$.

{\bf (1)}
Assume  there exists a constant $L > 0$ such that for every\/ $r \in (a,b)$ and every\/ $y_1 < y_2$ we have
\begin{equation}\label{Lipschitz}
g(r,y_2) - g(r,y_1) \le L \, (y_2 - y_1).
\end{equation}
Then,  $ ( \overline y - \underline y) (r_0) < 0 $ at some $r_0 \in (a,b)$ implies that
 $\displaystyle \sup_{(a,r_0]}  ( \overline y - \underline y )  <  0$.

\medskip

{\bf (2)}
If instead $g$ satisfies for every\/ $r \in (a,b)$ and every\/ $y_1 < y_2$ the reverse inequality 
\begin{equation}\label{Lipschitz2}
g(r,y_2) - g(r,y_1) \ge L \, (y_2 - y_1),
\end{equation}
then the positivity $(\overline y- \underline y)(r_0) >  0 $
at some $r_0 \in (a,b)$ implies $\displaystyle \inf_{[r_0, b)} (\overline y -  \underline y)  >  0$.
\end{lemma}
\begin{proof}
Before giving details, we remark that the proof relies on the following two observations
\begin{align}
\label{eq:Observation1} 
 \mbox{$z > 0$ and differentiable on $(t_0,t_1]$,}
 \quad\!
 \sup_{t \in (t_0, t_1) }      \frac{\, z'(t) \,}{z (t)}   < +\infty
 \quad  &\Rightarrow \quad\!
 \inf_{t \in (t_0, t_1]}  z (t) > 0,
\\
\noalign{\medskip}
\label{eq:Observation2} 
 \mbox{$z < 0$ and differentiable on $[t_0,t_1)$,}
 \quad\!
 \inf_{t \in (t_0, t_1) }      \frac{\, z'(t) \,}{z (t)}   > - \infty
 \quad  &\Rightarrow \quad\!
 \sup_{t \in [t_0, t_1) }  z (t) < 0 ,
\end{align}
which can respectively be justified by applying the mean value theorem to the function $\ln z $ on $[t_0 + \varepsilon, \, t_1]$ (respectively, $\ln (-z)$ on $[t_0, \, t_1 - \varepsilon]$).

\medskip

{\bf Proof of (1):} Consider the set 
\begin{equation} \label{eq:Electricity}
   I := \{\, r \in [a,r_0] \, : \,  \sup_{ (r,r_0]} ( \overline y - \underline y) < 0 \,\}  .
\end{equation} 
By continuity of $ \overline y - \underline y$, the set~$I$ contains $(r_0 - \varepsilon, \, r_0]$ and is open in $[a,r_0]$. More precisely, $I$~is an interval because if $r \in I$ then $(r,r_0] \subset I$. Hence we have either $I = [a,r_0]$ or $I = (t_0,r_0]$ for some $t_0 \in (a,r_0)$. However, in the last case we would also have $(\overline y - \underline y)(t_0) = 0$.
Furthermore, from~\eqref{Cauchy-Comparison} and by assumption~\eqref{Lipschitz} it follows that for each $r \in (t_0,r_0]$
$$
    ( \underline y - \overline y ) ' (r) \leq g( r,  \underline y (r) ) -  g( r,  \overline y (r) )  \leq L \, (\underline y - \overline y) (r) .
$$
Hence, by applying \eqref{eq:Observation1} to $z = \underline y  - \overline y$ in~$(t_0,r_0]$ we reach a contradiction. Therefore we must have $I = [a,r_0]$, which is equivalent to the claim.

\medskip

{\bf Proof of (2):}  
The steps are similar, but instead of \eqref{eq:Electricity} one must consider the set
$\displaystyle \{\, r \in [r_0, b] \, : \,  \inf_{ [r_0, r)} ( \overline y - \underline y) > 0 \,\}$. This set is a nonempty, open subinterval of~$[r_0,b]$ whose first endpoint is~$r_0$. The second endpoint cannot be less than~$b$ as a consequence of
assumption~\eqref{Lipschitz2} together with~\eqref{eq:Observation2} applied to $z = \underline y - \overline y$.
\end{proof}

In the present paper, such a comparison will be exploited by letting $g$ be the function $w(r, c)$ defined in \eqref{w}. 
Both conditions \eqref{Lipschitz} and \eqref{Lipschitz2} hold true
because $w$ is linear in~$c$, and satisfies
$$
           w (r, c_2) - w (r, c_1) = \alpha (r) \, (c_2 - c_1) 
$$
with $\displaystyle \alpha (r) := \Big( h(r)^{N-1} \, \int_{R_0}^r \frac{ds}{h(s)^{N-1}} \Big)^{-1}$ locally bounded for $r \in [R_0,S)$.
%
As a consequence, we deduce that given on some interval $[a,b] \subset [R_0, S)$ 
two functions $\overline y, \underline y \in C^1 ([a,b])$ satisfying
\begin{equation} \label{eq:CSI}
\left\{
   \begin{array}{c}
    \overline y'(r)\geq w(r, \overline y(r) ),
    \\
\noalign{\medskip}
    \underline y'(r) \leq w(r, \underline y(r) ),
   \end{array}
   \right.
\end{equation}
then 
\begin{equation} \label{eq:CSI2} 
   \left\{ 
   \begin{array}{c} 
    ( \overline y - \underline y)(a) \geq  0 
    \quad \Longrightarrow \quad \overline y - \underline y \geq 0  
    \, \hbox{ in } [a,b] ,
    \\
\noalign{\medskip}
    ( \overline y -  \underline y)(b) \leq 0 
    \quad \Longrightarrow \quad \overline y - \underline y \leq 0     
    \, \hbox{ in } [a,b] .
    \end{array} 
    \right.
\end{equation}

As discussed next, this comparison result turns out to be extremely useful to derive our results for the Bernoulli problem \eqref{pr2}. 
In fact we will use it to compare the radial solutions of problem \eqref{Dp2} with the Dirichlet boundary value $\varphi$ of problem \eqref{pr2}.

\subsection{Radial symmetry for the Bernoulli problem}

Similarly to what has been done in Section \ref{sec:rel_dic_prob}, the analysis of the radial symmetry for  the Bernoulli problem~\eqref{pr2}
is tackled by considering
the two geodesic balls $B_{R_i} (O)$  with radii given by
\begin{equation}\label{R_12}
R_1 = \min_{\partial \Omega} r > R_0\,, \qquad R_2 = \max_{\partial \Omega}\, r < S.
\end{equation}
We have $\overline B_{R_0}(O) \subset B_{R_1}(O)  \subset \Omega \subset B_{R_2}(O) \subset B_R(O)$ and from
the compactness of $\partial \Omega$ there exist $P_i \in \partial B_{R_i}(O) \cap \partial \Omega$.  
Thanks to Proposition~\ref{derivative_0}, for $i = 1,2$ we can consider the unique solutions $u_1,u_2 \in C^2(\mathcal{M} \setminus B_{R_0}(O))$ of
\begin{equation} \label{eq:Yorquis} 
  - \Delta u_i = f \mbox{ in $\mathcal{M} \setminus B_{R_0}(O)$}, 
  \qquad
  u_i = 0 \, \, \hbox{ on } \partial B_{R_0} ,
  \quad
  u_i = \varphi(R_i)  \, \, \hbox{ on } \partial B_{R_i}.
\end{equation} 
Since $u_1,u_2$ are radial, slightly abusing the notation we will consider them as functions of~$x$ as well as~$r$.
Following the arguments used in Lemma~\ref{lem:pasolini}, but now taking also into account that the functions $u, u_1, u_2$ vanish on $\partial B_{R_0} (O)$, we can likewise derive 
\begin{equation} \label{eq:MalTesta2}
 u \hbox{ is radial} 
  \, \,  \, \Longleftrightarrow \, \, \, u  \equiv u_i \hbox{ in\/ $\Omega  \setminus B_{R_0}  (O)$ for both\/ $i = 1,2$}. 
\end{equation} 

\medskip

The comparison stated above in \eqref{eq:CSI2} will then play a central role to justify that $u_1 \leq u \leq u_2$, which is one of the steps 
needed in the proof of our second main theorem.

\begin{proof}[Proof of Theorem \ref{th2}]
The argument is similar to the one in the proof of Theorem~\ref{th1}, but we need to take into account the presence of the ball $B_{R_0} (O)$. 
Consider the two functions $u_1, u_2$ defined by \eqref{eq:Yorquis}. We stress that $u_1(r)$ exists for all $r \in [R_0,S)$, not only for $r \in [R_0,R_1]$, which allows for a comparison with~$u$ along $\partial \Omega$ instead of~$\partial B_{R_1}(O)$. 
The proof is divided into three parts.

\medskip

\textit{Part 1.} Let us show that 
\begin{equation} \label{Unfinished}
     u_1 \leq u \leq u_2
    \quad
    \hbox{in } \overline \Omega \setminus B_{R_0}(O) ,
\end{equation}
as well as
\begin{equation} \label{Unfinished2}
     (u'_1 - u_{\nu}) (P_1) \geq 0
    \quad \hbox{and} \quad
      (u'_2 - u_{\nu}) (P_2) \leq 0.
\end{equation}

\medskip

 From \eqref{varphi2}, and taking into account the equality \eqref{key}  satisfied by $u_i$ we have
\[
\varphi'(r)  \geq w(r, \varphi(r)) ,
\quad
u_i'(r) = w(r,u_i(r))  ,
\quad
   u_i (R_i) = \varphi (R_i) .
\]
Hence, $(\varphi, u_i)$  satisfy the pair of inequalities \eqref{eq:CSI} and therefore from \eqref{eq:CSI2} we obtain 
$$
    (\varphi - u_1) (r) \geq 0 \hbox{ on }  (R_1,S),
    \qquad
    (\varphi - u_2) (r) \leq 0 \hbox{ on }  (R_0, R_2).
$$
As a result, we have
\[
 u_1 \leq 
 u = \varphi \leq u_2 
\quad \hbox{on } \partial \Omega .
\]
Since 
$\Delta (u -u_i)=0$ in $\Omega \setminus \overline B_{R_0}(O)$ with  $u - u_i=0$ on $\partial B_{R_0}(O)$ for $i  \in \{1,2\}$ (by construction of $u_i$) 
the maximum principle applied to $u - u_i$ gives~\eqref{Unfinished}.

\medskip

Furthermore, on the one hand we have $u(P_i) = u_i(R_i)$. On the other hand, we 
 observe that the boundaries of $\Omega$ and $B_{R_i} (O)$ are tangent at $P_i$, 
  and the outer normal~$\nu$ of~$\partial \Omega$ 
 at~$P_i$ coincides with the radial direction. Hence
\begin{equation}\label{inequalities}
u_1'(R_1) \ge u_\nu(P_1),
  \qquad 
u_2'(R_2) \le u_\nu(P_2),
\end{equation}
which gives the conclusion \eqref{Unfinished2}. 

\bigskip

\textit{Part 2.} 
Let us  prove that
\begin{equation}\label{zero}
 w(r_0,\varphi(r_0)) -  \nd(r_0) 
 =0 ,
 \quad
 \hbox{ for some } r_0 \in  [R_1,R_2].
\end{equation}
Recalling that $u_\nu(P_i) = \nd (R_i)$ by the boundary condition, and using the relation \eqref{key} satisfied by $u'_i$, the inequalities \eqref{inequalities} may be rewritten as
\[
 w(R_1,\varphi(R_1)) \ge \nd (R_1),
\qquad
 w(R_2,\varphi(R_2)) \le \nd (R_2),
\]
and therefore \eqref{zero} must hold.

\medskip

Let us check that $u$ is radial. Recalling \eqref{Unfinished}, by the strong maximum principle and \eqref{eq:MalTesta2} we have the alternative
\begin{equation} \label{eq:Alternativa-2} 
\left\{
\begin{array}{l}
\hbox{either $u_1 \equiv u \equiv u_2 $ in $\Omega \setminus B_{R_0} (O)$, and $u$ is radial};
\\
\noalign{\medskip}
\hbox{or $u_1 < u <  u_2 $ in $\Omega \setminus \overline B_{R_0} (O)$, and $u$ is not radial}.
\end{array}
\right.
\end{equation}

Now, on account of the definition of the ball $B_{R_1}(O)$, the interior sphere condition is satisfied at the point $P_1$.
Hence, by the Hopf Lemma we can rephrase the alternative \eqref{eq:Alternativa-2} as follows:

\begin{enumerate} \itemsep=5pt
\item[{\bf (i)}]
either $w (R_1, \varphi(R_1))- \nd(R_1) =0$ and in this case $u_1 \equiv u \equiv u_2$ in $\overline \Omega \setminus B_{R_0}(O)$;
\item[{\bf (ii)}]
or $w (R_1, \varphi(R_1))- \nd(R_1) > 0 \geq w (R_2, \varphi(R_2))- \nd(R_2)$ and $u$ is not radial.
\end{enumerate} 

\medskip
If the assumption (1) of the theorem holds, from \eqref{zero} we deduce that we are in case~{\bf (i)},
and so the solution $u$ is radial in $\Omega$. 

If the interior sphere condition is in addition assumed at each point of $\Omega$, then the Hopf Lemma holds at~$R_2$. Hence
the hypothesis (2) of the theorem, combined with~\eqref{zero} and the alternative above, gives that   
\[
w(R_2,\varphi(R_2)) - \nd (R_2) = 0 
 \quad \hbox{ and } \quad
 u_1  \equiv u \equiv u_2 \, \hbox{ in } \Omega,
\]
so $u$ is radial.

\medskip

\textit{Part 3.} 
Let us prove that $\Omega$ is a geodesic ball centered at $O$ if one of the assumptions (a), (b) or (c) holds.
Suppose we are in case (1) or (2), so that we have $u \equiv u_1 \equiv u_2$ in~$\overline \Omega \setminus B_{R_0}(O)$. Since the outer unit vector~$\nu$ to~$\partial \Omega$ coincides with~$\frac\partial{\, \partial r \,}$ at~$P_1$ and~$P_2$, the normal derivative $u_\nu(P_i)$ coincides with $w(R_i,\varphi(R_i))$, hence
\begin{equation}\label{equalities}
w(R_i,\varphi(R_i)) - \kappa(R_i) = 0
\quad
\mbox{for $i = 1,2$.}
\end{equation}
Assume, by contradiction, that $R_1 < R_2$.
Observe that the boundary values of~$u$ obviously satisfy $\varphi(r) = u_1(r)=u_2(r)$ for $r \in [R_1, R_2]$, and therefore $\varphi'(r) = w(r,u_i(r))$ for $r \in [R_1, R_2]$ and $i = 1,2$ as a consequence of~\eqref{key}. Thus, we have equality in \eqref{varphi2}.

If assumption~(a) holds, we immediately reach a contradiction, and therefore we must have $R_1 = R_2$.

Let us treat the case when either assumption (b) or (c) holds.
We will write $u(r)$ as well as $u(x)$ because $u$ is radial. We observe that the difference $w(r,\varphi(r)) - \nd(r)$ vanishes identically in $[R_1,R_2]$ because it is monotone by assumption and vanishes at the endpoints by~\eqref{equalities}. This means that
the normal derivative of~$u$ on~$\partial \Omega$, which is~(\ref{gradient2}),
coincides with the radial derivative $u'(r) = w(r,\varphi(r))$.
Moreover, since $u$ is radial, we may write
\[
\nabla u(x) = w(r,\varphi(r)) \, \frac\partial{\, \partial r \,},
\quad
\hbox{for all } 
x \in \partial \Omega.
\]
Since 
\[
\nabla u(x) \cdot \nu = \nd(r),
\quad
\hbox{for all } 
x \in \partial \Omega\,,
\]
we obtain
\begin{equation} \label{eq:Klug2}
w(r,\varphi(r)) \Big(\frac\partial{\, \partial r \,} \cdot \nu - 1 \Big) =0
\quad
\hbox{for all } 
x \in \partial \Omega.
\end{equation}

Now, if (b) holds, then $w(r,\varphi(r)) = \nd(r) \ne 0$ in $[R_1,R_2]$, and from~(\ref{eq:Klug2}) we get $\frac\partial{\, \partial r \,} \cdot \nu = 1$. Hence $\nu \equiv \frac\partial{\, \partial r \,}$ on~$\partial \Omega$, and the conclusion follows. 

Finally, suppose that (c)~holds. Since we are supposing $R_1 < R_2$, there exists $P \in \partial \Omega$ such that $\nu$ is not the radial direction. Denote by $N(P) \subset \partial \Omega$ a neighborhood of~$P$ where $\nu$ keeps different from~$\frac\partial{\, \partial r \,}$. We obviously have $\frac\partial{\partial r} \cdot \nu < 1$ in $N(P)$, and therefore \eqref{eq:Klug2} implies $u'(r) = 0$ in an interval $(a,b) \subset (R_1, R_2)$. Hence by~\eqref{integrale} we have
\[
\int_{R_0}^r h(t)^{N-1} \, f(t)\, dt =
h(R_0)^{N-1} \, u'(R_0)
\]
for $r \in (a,b)$. Differentiating with respect to~$r$ we obtain $f \equiv 0$ in $(a,b)$, which contradicts assumption~(c). Hence we must have $R_1 = R_2$ once again. This concludes the proof of Theorem \ref{th2}.
\end{proof}

We end this subsection with a criterion to guarantee the validity of assumption ~(1) in Theorem~\ref{th2}. In fact, such assumption follows as soon as we can control from above the rate of increase of the product $h(r)^{N-1} \, \nd (r)$ in terms of the function~$f$. 
More precisely, it is enough that the inequality
\begin{equation}\label{q}
h(r_2)^{N-1} \, \nd (r_2) - h(r_1)^{N-1} \, \nd (r_1)
+
\int_{r_1}^{r_2} h(t)^{N-1} \, f(t) \, dt
\le 0
\end{equation}
holds for every $r_1,r_2 \in (R_0,S)$ with $r_1 < r_2$.

\begin{proposition}\label{explanation}
If\/ $\varphi$ satisfies\/~{\rm(\ref{varphi2})} and\/ $\nd$ satisfies\/~{\rm(\ref{q})}, then condition\/~{\rm(1)} of Theorem~\ref{th2} holds true.
\end{proposition}
\begin{proof}
Suppose that for some $r_0 \in (R_0, S)$ we have $w(r_0 ,\varphi(r_0)) - \nd (r_0)=0$. We must prove that
 the inequality $w(r,\varphi(r)) - \nd (r) \leq 0$ holds for every $r \in (R_0, r_0]$.

\medskip

Let $u_0$ denote the radial solution of the Dirichlet problem~\eqref{Dp2} with $R = r_0$ and $c = \varphi(r_0)$. By virtue of~\eqref{key} we have 
\begin{equation}\label{equivalent}
   u'_0(r_0) =  w(r_0 ,\varphi(r_0))  = \nd (r_0).
\end{equation}
Integrating~\eqref{rewritten} on $[r, r_0]$ we find
$$
   ( h^{N-1} \, u'_0) (r_0)  -  ( h^{N-1} \, u'_0) (r)  = - \int_{r}^{r_0} h^{N-1} (t) \, f(t) \, dt,
$$
whereas from assumption \eqref{q} we get
$$
   ( h^{N-1} \, \nd ) (r_0)  -  ( h^{N-1} \, \nd)  (r)  \leq - \int_{r}^{r_0}. h^{N-1} (t) \, f(t) \, dt.
$$
By subtracting and using \eqref{equivalent}, 
we deduce that  $h^{N-1} \, ( u_0'  - \nd) \leq 0$ on $(R_0, r_0]$ and therefore we obtain
\begin{equation} \label{eq:one-SideBis}
 \nd (r) \geq u_0'(r) \stackrel{\eqref{key}}{=} w \big(r , u_0 (r) \big)  \quad \mbox{on $(R_0, r_0]$}.
\end{equation}
On the other hand,  since $u_0 (r_0) = \varphi (r_0)$ and $(\varphi, u_0)$ satisfy the pair of inequalities~\eqref{eq:CSI} (by assumption~\eqref{varphi2} and  the equation~\eqref{key}), we deduce from
\eqref{eq:CSI2} that  $\varphi(r) \leq u_0 (r) $ for $r \in (R_0, r_0]$. Therefore, since $w$~is increasing in the second argument, inequality~\eqref{eq:one-SideBis} leads to
$$
   \nd (r) \geq w (r, u_0 (r)) \geq w (r, \varphi (r)) 
$$
for all $r \in (R_0, r_0]$, as claimed.
\end{proof}

\medskip

\begin{remark}~
\begin{enumerate} \itemsep=4pt
\item[(a)]
When $f \equiv 0$, condition \eqref{q} is equivalent to requiring the function $h^{N-1} \, \nd$ to be monotone non-increasing in the interval $(R_0, S)$.
\vspace{1mm}
\item[(b)]
When $\nd \in C^1$, after dividing by $r_2 - r_1$ and letting tend $r_1 \to r_2^-$, we can see that 
condition~\eqref{q} is equivalent to
\[
   -  \big( h^{N-1} \, \nd \big)' \geq h^{N-1} \, f.
\]
Such an inequality can equivalently be rephrased by saying that any antiderivative of~$\nd$ is a supersolution to the ODE~\eqref{rewritten} (i.e., a radial supersolution to~(\ref{Poisson})). 
\end{enumerate} 
\end{remark} 

\medskip
\subsection{The case of constant boundary data}
In this last subsection we focus on problem \eqref{pr2} in the special and important case when $\varphi$ is a constant. In other words, we consider the Bernoulli problem 
\begin{equation} \label{eq:ConstantData}
	\begin{cases}
		- \Delta {u}=f(r) &\mbox{in $\Omega \setminus \overline B_{R_0}(O)$},\\
		u=0 &\mbox{on $\partial B_{R_0}(O)$},\\
		u=c &\mbox{on $\partial \Omega$},\\
		u_\nu = \nd (r) &\mbox{on $\partial \Omega$},
	\end{cases}
\end{equation}
where $c$ is a prescribed constant.
Firstly, as a special case of Theorem~\ref{th2} we mention the following result:

\begin{corollary}\label{annulus}
Let\/ $\Omega \subset \mathcal{M}$ be a $C^1$ domain satisfying \eqref{eq:Ipotesi3}. Consider 
 $c \in \mathbb R$, and two functions $f \in C^0( [ R_0,S))$ and $\nd \in C^0( (R_0,S))$ satisfying \eqref{q} for which 
the overdetermined problem~\eqref{eq:ConstantData} has a solution\/~$u$.
If\/ $c \leq 0$ and $f \geq 0$, then $u$ is radial; if furthermore\/ $c <0$, then\/ $\Omega$ is a geodesic ball centered at $O$. 
 \end{corollary}
\begin{proof}
Referring to the definition \eqref{w}, note that for $\varphi \equiv c \leq 0$ and $f \geq 0$, we have
\begin{eqnarray*}
  \varphi' (r) - w (r, \varphi(r)) 
  &=&
  {} - w(r,c) 
 \nonumber \\
 &=&
 {} - \frac1{\, h(r)^{N-1} \, \int_{R_0}^r \frac{ds}{h(s)^{N-1}} \,}
\left( c - \int_{R_0}^r \int_{s}^r \Big( \frac{h(t)}{h(s)} \Big)^{N-1} \, f(t) \, dt \, ds \right)
\\
&\geq& 0.
\end{eqnarray*}
Therefore the differential inequality~\eqref{varphi2} is satisfied, and this inequality is strict when $c <0$. 
By Proposition \ref{explanation}, assumption (1) of Theorem \ref{th2} is also satisfied and the conclusion follows.
\end{proof}

\medskip

We now discuss how such a result can be derived  in the case $f \leq 0$ for solutions that satisfy $u \geq c$. To reach this goal, we recall
an observation that was done in \cite[Corollary~3.2]{greco_JAM} for the case when $\mathcal{M} = \R^N$, allowing to compare $u$ and~$u_2$ without any sign assumptions on $f$.

\begin{lemma}[Comparison]\label{comparison}
Let\/ $\Omega \subset \mathcal{M}$ be a\/ $C^1$ domain satisfying \eqref{eq:Ipotesi3}, 
$f \in C^0([R_0, \allowbreak S))$ and consider a constant $c \leq 0 $. Let $u$ be a solution of the Dirichlet problem
\[
\begin{cases}
-\Delta u = f(r) &\mbox{in~$\Omega \setminus \overline B_{R_0}(O)$,}
\\
u = 0 &\mbox{on~$\partial B_{R_0}$, }
\\
u = c &\mbox{on~$\partial \Omega$,}
\end{cases}
\]
and let\/ $u_2$ be the radial solution of
\[
\begin{cases}
-\Delta u = f(r) &\mbox{in~$B_{R_2}(O) \setminus \overline B_{R_0}(O)$, }
\\
u = 0 &\mbox{on~$\partial B_{R_0}$, }
\\
u = c &\mbox{on~$\partial B_{R_2}$,}
\end{cases}
\]
with $R_2$ defined as in~\eqref{R_12}. 
If\/ $u \ge c$ in\/ $\Omega \setminus \overline B_{R_0}(O)$, then\/ $u \le u_2$ in\/ $\Omega \setminus  B_{R_0}(O)$.
\end{lemma}

\begin{proof}
 If we know in advance that $u \le u_2$ on~$\partial \Omega$, then the conclusion is a consequence of the maximum principle applied to the harmonic function 
 $\omega := u - u_2$ in $\Omega \setminus \overline B_{R_0}(O)$. The purpose of the lemma is to replace such a boundary inequality with the assumption $u \ge c$ in $\Omega \setminus \overline B_{R_0}(O)$. 

\medskip

Let us set
$\displaystyle
M := \max_{\overline \Omega \setminus B_{R_0}(O)} \omega$ which, by the maximum principle, is attained at some $P_0 \in \partial \Omega \cup \partial B_{R_0}$.
Denoting by $r_0 \in \allowbreak [R_0, R_2]$ the radial coordinate of~$P_0$, we distinguish two cases.

Case (i): $r_0 \in \{ R_0, R_2\}$.
In this case, $P_0 \in \partial B_{R_0} (O)$ or $P_0 \in \partial \Omega \cap \partial B_{R_2} (O) $, and from the definition of $u_2$ we deduce that $\omega (P_0) = 0$. Namely, 
$M = 0$ and so $u - u_2 \leq 0$ as claimed.

Case (ii): $r_0 \in (R_0, R_2)$.
In this case we have $P_0 \in \partial \Omega$ and $u(P_0) =c$. From the fact that  $\Omega$ is connected, we 
observe that the set $\Sigma_{r_0} = \Omega \cap \allowbreak \partial B_{r_0}(O)$ is nonempty. 
Considering then any $P \in \Sigma_{r_0} $ (whose radial coordinate is $r_0$),  from the assumption $u \ge c$ and the fact that $u_2$ is radial, we get
\[
\omega(P) =  u(P) - u_2 (r_0)  \ge c - u_2(r_0)  = u(P_0) - u_2 (r_0)  = \omega(P_0) = M.
\]
Applying the strong maximum principle, we deduce in this case that the harmonic function~$\omega \equiv M$ in $\overline \Omega \setminus B_{R_0}$. So
$\omega \equiv 0$ (since it vanishes on $\partial B_{R_0} (O)$) and the claim also holds in that case.
\end{proof}

Theorem~\ref{th2} together with the previous lemma allow us to show the following corollary, whose special case of $\mathcal{M}=\RN$ follows also from \cite[Theorem~1.3]{greco_JAM}.

\begin{corollary}
Let\/ $\Omega \subset \mathcal{M}$ be a $C^1$ domain satisfying~\eqref{eq:Ipotesi3}. Consider a constant $c$ and two functions $f \in C^0([R_0,S))$, $\nd \in C^1((R_0,S))$ such that
\begin{equation}\label{old}
{} - \frac{d}{\, dr \,} \Big( h(r)^{N-1}  \, \nd (r) \Big)
\ge
h(r)^{N-1} \, f(r),
\quad
f \leq 0 , 
\quad
c \leq 0 .
\end{equation}
If  the overdetermined problem\/~\eqref{eq:ConstantData} admits a solution\/~$u$
satisfying $u \geq c$ in\/ $\Omega \setminus B_{R_0}(O)$, 
then $u$ is radial, and if in addition\/ $c <0$ then the domain $\Omega$ is a geodesic ball centered at the origin. 
\end{corollary}
\begin{proof}
Note that the first condition in \eqref{old} implies the validity of~\eqref{q} (by integrating), and so Proposition~\ref{explanation} guarantees the hypothesis (1) of Theorem~\ref{th2}. 
Let us check that the condition~\eqref{varphi2} also holds.

\medskip

By Lemma~\ref{comparison} we know that $u_2 \ge u \geq c$ in $\Omega \setminus B_{R_0}(O)$, hence for every $P_2 \in \partial B_{R_2} \cap \partial \Omega$ we have $u_2'(R_2) \le u_\nu(P_2) \le 0$. 
Since $f \le 0$, the ODE~\eqref{rewritten} shows that the product $r^{N-1} \, u_2'(r)$ is non-decreasing, hence $u'_2 (r) \leq 0$ for all $r \in [R_0,R_2]$. 
Therefore, by the monotonicity of $w$ in the second argument, we have
\begin{eqnarray*}
  \varphi' (r) - w (r, \varphi(r)) 
  &=&
  {} - w(r,c) 
 \nonumber \\
 &\geq&
 {} - w(r, u_2(r))
     \nonumber \\
 &\stackrel{\eqref{key}}{=}& {} - u'_2(r) 
      \nonumber \\
&\geq&
 0
\end{eqnarray*}
and the inequality is strict when $c <0$. Consequently, the conclusion follows from Theorem~\ref{th2}.
\end{proof}

\end{document}